\documentclass[12pt]{amsart}
\usepackage[final]{epsfig}
\usepackage{graphics}
\usepackage{amsmath}
\usepackage{amsfonts}
\usepackage{latexsym}
\usepackage{amssymb}
\usepackage{graphicx}
\usepackage{url}
\usepackage[margin=1in]{geometry}

\newtheorem*{rep@theorem}{\rep@title}
\newcommand{\newreptheorem}[2]{%
\newenvironment{rep#1}[1]{%
 \def\rep@title{#2 \ref{##1}}%
 \begin{rep@theorem}}%
 {\end{rep@theorem}}}
\makeatother

\newtheorem{lemma}{Lemma}
\newtheorem{proposition}[lemma]{Proposition}
\newtheorem{theorem}{Theorem}
\newtheorem{corollary}[lemma]{Corollary}
\theoremstyle{definition}

\newtheorem{definition}{Definition}
\newtheorem{question}{Question}
\newtheorem{example}{Example}
\newtheorem{remark}[lemma]{Remark}
\newtheorem*{theorem*}{Theorem}
\newreptheorem{theorem}{Theorem}

\newcommand{\Q}{{\mathbb Q}}

\newcommand{\N}{{\mathbb N}}
\newcommand{\Z}{{\mathbb Z}}

\newcommand{\lcm}{{\rm lcm}}
\newcommand{\ord}{{\rm ord}}
\newcommand{\denom}{{\rm denom}}
\newcommand{\numer}{{\rm numer}}

\title[The $p\, $--adic Order of Power Sums and the Erd\H{o}s\,--Moser Equation]{The $\boldsymbol{p}\,$--adic Order of Power Sums,\\the Erd\H{o}s\,--Moser Equation, and Bernoulli Numbers}

\begin{document}

\author{Jonathan Sondow}
\address{209 West 97th Street, New York, NY 10025}
\email{jsondow@alumni.princeton.edu}
\author{Emmanuel Tsukerman}
\address{Department of Mathematics, University of California, Berkeley, CA 94720-3840}
\email{e.tsukerman@berkeley.edu}

\date{}
\maketitle

\begin{abstract}
The \emph{Erd\H{o}s--Moser equation} is a Diophantine equation proposed more than 60 years ago which remains unresolved to this day. In this paper, we consider the problem in terms of divisibility of power sums and in terms of certain Egyptian fraction equations. As a consequence, we show that solutions must satisfy strong divisibility properties and a restrictive Egyptian fraction equation. Our studies lead us to results on the Bernoulli numbers and allow us to motivate Moser's original approach to the problem. 
\end{abstract}

\begin{center}\section*{Table of Contents}\end{center}

\begin{itemize}
  \item[] Section 1. Introduction
  \item[] Section 2. Main Results
  \item[] Section 3. Power Sums
  \item[] Section 4. Egyptian Fraction Equations
  \item[] Section 5. Bernoulli Numbers
  \item[] Section 6. Moser's Mathemagical Rabbits
\end{itemize}

\section{Introduction}

For $m\in\N$ and $n\in\Z$, define the \emph{power sum}
\[S_n(m):=\sum_{j=1}^{m}j^n = 1^n+2^n+\dotsb +m^n\]
and set $S_n(0):=0$. The \emph{Erd\H{o}s--Moser equation} is the Diophantine equation
\begin{equation} \label{EQ:EME}
S_{n}(m)=(m+1)^n.
\end{equation}
Erd\H{o}s and Moser \cite{Moser} conjectured that the only solution is the trivial solution $1+2=3$, that is, $(m,n)=(2,1)$. See Moree's surveys ``A top hat for Moser's four mathemagical rabbits''~\cite{MoreeTopHat} and \cite{MoreeOnMoser}, as well as Guy's discussion in~\cite[Section D7]{Guy}.

The \emph{generalized Erd\H{o}s--Moser equation} is the Diophantine equation
\begin{equation} \label{EQ:GEME}
S_{n}(m)=a(m+1)^n.
\end{equation}
Moree \cite{Moree} conjectured that the only solution is the trivial solution
\[
1+2+3+\dots+2a=a(2a+1),
\]
that is, $(m,n)=(2a,1)$.

In this paper, we consider the equations from two angles: as problems on the divisibility of power sums and as problems on Egyptian fraction equations. In the final two sections, we consider implications of our results to the Bernoulli numbers, and motivate Moser's ``mathemagical rabbits.''

\section{Main Results}

 For $q\in\Z$ and prime~$p$, the \emph{$p$-adic order of $q$} is the exponent $ v_p(q)$ of the highest power of $p$ that divides~$q$:
 $$v_p :\Z\to\N\cup\{0,\infty\}, \qquad v_p(q):=\sup_{p^d\mid q}d.$$
We note that the domain of definition of $v_p$ can be extended to the $p$-adic integers ${\Z}_p\supset \Z$ by considering the digits of the base $p$ expansion.

We also define a map
\[
V_p: \mathbb{Z}_p \rightarrow \N\cup \{0,\infty\}, \qquad V_p(m):=v_p(m-\lfloor\frac{m}{p}\rfloor)+1.
\]
 This function can be interpreted as follows: $V_p(m)$ counts the number of equal $p$-digits at the end of the base $p$ expansion of $m \in {\Z}_p$. That is, if we write $m$ in base~$p$ as
\[
m={\dots a_{k} \dots a_1 a_0}_p=\sum_{i=0}^{\infty} a_i p^i,
\]
and let 
\[
h=\sup\{i \in \N\cup \{0\}: a_i=a_j  \ \forall \ 0 \leq j \leq i\},
\]
then $V_p(m)=h+1$.

\begin{reptheorem}{porderOfPowerSums}
Let $p$ be an odd prime and let $m$ be a positive integer.\\
{\rm(i).} In case $m\equiv0$ or $-1\pmod{p}$, we have
\[
v_p(S_n(m)) \begin{cases} =v_p(S_{p-1}(m))=V_p(m)-1\quad \text{ if }p-1\mid n,\\
                                                 \geq V_p(m)\qquad\qquad\qquad\qquad\, \quad\text{ if }p-1\nmid n.
\end{cases}
\]
{\rm(ii).} In case $m\equiv \frac{p-1}{2} \pmod{p}$, we have
\[
v_p(S_n(m)) \begin{cases} =v_p(S_{p-1}(m))=V_p(m)-1\quad \text{ if $n$ is even},\\
                                                 \geq V_p(m)\qquad\qquad\qquad\qquad\, \quad\text{ if $n$ is odd.}
\end{cases}
\]
\end{reptheorem}

As a result, we prove:

\begin{reptheorem}{thm: EMresults}
Let $p$ be an odd prime. \\
{\rm(i).} In the generalized Erd\H{o}s--Moser equation, if $p\mid m+1$, then $p-1 \nmid n$.\\
{\rm(ii).} In the Erd\H{o}s--Moser equation, if $p\mid m$, then $p-1 \mid n$ and $p^2\mid m+p$. Also, if $p\mid m- \frac{p-1}{2}$, then $p-1 \mid n$ and $m \equiv   -(p+\frac{1}{2}) \pmod{p^2}$.
\end{reptheorem}

Next we consider Egyptian fraction equations of the following form. For a given positive integer $n$, we seek an integer $d$ so that the congruence
\begin{align} \label{EQ:coolNums1}
\sum_{p \mid n} \frac{1}{p}+\frac{d}{n} \equiv 1 \pmod{1}
\end{align}
holds. Integers $n$ for which $d \equiv \pm 1 \pmod{n}$ are closely related to Giuga numbers and primary pseudoperfect numbers. Moreover, $d$ as a function of $n$ can be seen as an arithmetic derivative of $n$ and is related to the arithmetic derivative considered in \cite{Barbeau, Churchill, UfnarovskiAhlander}. By studying these equations, we prove:

\begin{reptheorem}{thm: Rabbit}
Let $(m,n)$ be a nontrivial solution to the Erd\H{o}s--Moser equation.\\
{\rm(i).} The pair $(n,d)=(m-1,2^n-1-X)$ satisfies congruence \eqref{EQ:coolNums1}, where$$X:=\sum_{\substack{p \mid m-1,\\p-1 \nmid n}} \frac{m-1}{p}.$$\\
{\rm(ii).} If $p\mid m-1$, then $n=p-1+k\cdot \ord_p(2)$ for some $k\geq 0$.\\
{\rm(iii).} Given $p\mid m-1$, if $p^e \mid m-1$ with $e \geq 1$, then $p^{e-1} \mid 2^n-1$.\\
{\rm(iv).} Given $p\mid m-1$, if $p-1 \mid n$ and $p^{e} \mid 2^n-1$ with $e \geq 1$, then $p^{e+1} \mid m-1 $; in particular, $p^2 \mid m-1$.
\end{reptheorem}

As an application, combining this with the result of \cite{MoreeOnMoser} that $3^5 \mid n$, we see that if $(m,n)$ is a solution of the Erd\H{o}s--Moser equation with $m \equiv 1 \pmod{3}$, then in fact $m \equiv 1 \pmod{3^7}$.

 %%%
\section{Power Sums}

In all the formulas of this paper, \emph{the letter $p$ denotes a prime number}, unless ``integer $p$'' is specified.

For $m\in\N$ and $n\in\Z$, define the \emph{power sum}
\[S_n(m):=\sum_{j=1}^{m}j^n = 1^n+2^n+\dotsb +m^n\]
and set $S_n(0):=0$. 
Fixing a prime $p$, we define the \emph{restricted power sum} $S^*_{n}(0):=0$ and
\[S^*_{n}(m)=S^*_{n}(m,p):=\sum_{\substack{j=1,\\ (j,p)=1}}^{m}j^n,\]
obtained from $S_n(m)$ by removing the terms $j^n$ with $j$ divisible by~$p$. (Compare \cite[equation (2.1)]{EL}.) For example,
$S_{n}(p)-p^n=S_{n}(p-1)=S^*_{n}(p-1)=S^*_{n}(p)$
and, by induction on $d \in \N$,
\begin{align} \label{EQ:S*}
S_n(p^d)=S^*_{n}(p^d)+p^n S^*_{n}(p^{d-1})+p^{2n}S^*_{n}(p^{d-2})+\dotsb+p^{dn} S^*_{n}(p^0).
\end{align}

We now prove the linearity of certain restricted and unrestricted power sums upon reduction modulo prime powers.

\begin{lemma} \label{LEM:linear}
If $p$ is a prime, $d,q \in \N$, $N \in \Z$,  $m_1\in p^d \N\cup\{0\}$, and $m_2 \in \N\cup\{0\}$, then
\begin{align*}
S^*_{n}(qm_1+m_2) \equiv q S^*_{n}(m_1)+S^*_{n}(m_2) \pmod{p^d}.
\end{align*}
Furthermore, the congruence also holds with all $S_{n}^*$ replaced by the unrestricted sum $S_{n}.$
%\begin{align*}
%S_{n}(qm_1+m_2) \equiv q S_{n}(m_1)+S_{n}(m_2) \pmod{p^d}.
%\end{align*}
\end{lemma} 

\begin{proof}
Note first that
\begin{align*}
S^*_{n}(qp^d) &= \sum_{\substack{k=1,\\ (k,p)=1}}^{p^dq} k^{n} = \sum_{j=0}^{q-1} \sum_{\substack{k=1,\\ (k,p)=1}}^{p^d} (p^d j+k)^{n} 
 \equiv  q  \sum_{\substack{k=1,\\ (k,p)=1}}^{p^d} k^{n} \equiv qS^*_{n}(p^d)\pmod{p^d}.
\end{align*}
Since $m_1\in p^d \N\cup\{0\}$, we therefore have
\begin{align*}
 S^*_{n}(qm_1+m_2) =\sum_{\substack{k=1,\\ (k,p)=1}}^{qm_1+m_2}k^n &= \sum_{\substack{k=1,\\ (k,p)=1}}^{qm_1}k^n+\sum_{\substack{j=1,\\ (j,p)=1}}^{m_2}(qm_1+j)^n \\
& \equiv  q S^*_{n}(m_1)+S^*_{n}(m_2) \pmod{p^d},
\end{align*}
as desired. The proof in the unrestricted case is similar.
\qedhere
\end{proof}

\begin{theorem} \label{THM:PowerSumporder}
Let $p$ be an odd prime, and assume $d,q\in\N$ and $n \in \Z$. Then
\[
S^*_{n}(p^dq) \equiv\begin{cases}-p^{d-1}q \pmod{p^d} \, \quad\text{ if }p-1\mid n,\\
                                        \ 0\quad\quad\ \pmod{p^d}\quad \text{ if }p-1\nmid n.
                                         \end{cases}
\]
\end{theorem}

\begin{proof}
By Lemma \ref{LEM:linear}, it suffices to prove the theorem in the special case where $q=1$. Let $\phi(n)$ denote Euler's totient function.
Since$$S^*_{0}(p^d) = \phi(p^d) =p^{d-1}(p-1) \equiv -p^{d-1} \pmod{p^d}$$and $p-1\mid 0$, the result holds when $n=0$.

Now assume $n\neq0$. As $d>0$ and $p$ is an odd prime, $p^d$ has a primitive root~$g$. Then $g$ has multiplicative order $\phi(p^d)$ modulo $p^d$, and $g^n\neq1$. Hence $S^*_{n}(p^d)$ is congruent to
\begin{equation*}\label{EQ:primive root}
\begin{aligned}
S^*_{n}(p^d) &= \sum_{\substack{j=1,\\ (j,p)=1}}^{p^d}j^n \equiv  \sum_{i=0}^{\phi(p^d)-1}(g^i)^n 
&\equiv\sum_{i=0}^{\phi(p^d)-1}(g^n)^i  \equiv \frac{g^{n\phi(p^d)}-1}{g^n-1}\pmod{p^d}.
\end{aligned}
\end{equation*}

We now consider the case $n>0$.  If $p-1\mid n$, then Fermat's Little Theorem implies $g^n= 1+k p$, for some $k > 0$. Hence
\begin{align*}
S^*_{n}(p^d) 
 & \equiv \frac{(1+kp)^{\phi(p^d)}-1}{(1+kp)-1} \equiv \frac{\phi(p^d)kp+\binom{\phi(p^d)}{2}(kp)^2+\dotsb +(kp)^{\phi(p^d)}}{kp} \equiv \phi(p^d)\pmod{p^d}.
\end{align*}
Thus $S^*_{n}(p^d) \equiv -p^{d-1} \pmod{p^d}$, as desired. This proves the result when $p-1 \mid n$.

If $p-1\nmid n$, then a fortiori $\phi(p^d)\nmid n$, and so $g^n\not\equiv 1\pmod{p}$. As $g^{\phi(p^d)}\equiv 1\pmod{p^d}$,
\begin{align*}
S^*_{n}(p^d) \equiv \frac{g^{n\phi(p^d)}-1}{g^n-1}\equiv 0\pmod{p^d}.
\end{align*}
This proves the result when $p-1 \nmid n$, and the proof of the case $n>0$ is complete.

The case $n < 0$ follows, because another primitive root of $p^d$ is $g^{ \phi(p^d)-1} \equiv g^{-1} \pmod{p^d}$, and so $S^*_{n}(p^d) \equiv S^*_{- n}(p^d)\pmod{p^d}.$ This completes the proof of the theorem.
\end{proof}

In the following application of Theorem \ref{THM:PowerSumporder}, the case $d=q=1$ is classical. (For a recent elementary proof of that case, as well as a survey of other proofs and applications of it, see~\cite{MS1}.)

\begin{corollary} \label{PowerSumporder}
Let $p$ be an odd prime and $n,d,q \in \N$. Then
\[
1^n+2^n+\dotsb +(p^dq)^n \equiv \begin{cases}-p^{d-1}q \,\pmod{p^d}\quad \text{ if }p-1\mid n,\\
                                        \ 0\quad\quad\  \pmod{p^d} \quad\text{ if }p-1\nmid n.
                                         \end{cases}
\]
In particular, $S_n(p^dq)\equiv 0\pmod{p}$ if $d>1$ or $p>n+1$.
\end{corollary}

\begin{proof}
By the linearity of Lemma \ref{LEM:linear}, it suffices to prove the result when $q=1$. In case $n=1$, then $p-1\nmid n$ and $S_n(p^d)= p^d(p^d+1)/2 \equiv 0 \pmod{p^{d}}$, verifying this case. For $n>1$, we reduce both sides of equation \eqref{EQ:S*} modulo $p^d$, then apply Theorem~\ref{THM:PowerSumporder} to each term on the right-hand side, obtaining
%\[
%S_n(p^d)=S^*_{n}(p^d)+p^n S^*_{n}(p^{d-1})+p^{2n}S^*_{n}(p^{d-2})+\dotsb+p^{dn} S^*_{n}(p^0).
%\]
$S_n(p^d) \equiv S^*_{n}(p^d)\pmod{p^d}.$ The corollary follows.
\end{proof}

For example, taking $q=p$ gives$$S_{n}(p^{d+1})   \equiv 0 \pmod{p^d},$$whether or not $p-1$ divides $n$. For instance, $S_{2}(9) =285 \equiv 0\pmod{3}$ and $S_{1}(9) = 45 \equiv 0\pmod{3}$.

On the other hand, taking $q=1$ and replacing $d$ with $d+1$ in Corollary~\ref{PowerSumporder} gives$$p-1\mid n \implies S_{n}(p^{d+1})  \not\equiv0\pmod{p^{d+1}}.$$ For example, $S_{2}(9) =285\not \equiv 0\pmod{9}$.

\begin{corollary} \label{NegExpPowerSumporder}
For $n \in \N$,
\[
\text{prime } p\ge n+2 \implies \frac{1}{1^n}+\frac{1}{2^n} +\dotsb+\frac{1}{(p-1)^n} \equiv 0 \pmod{p}.
\]
\end{corollary}

\begin{proof}
The sum is $S_{-n}(p-1) = S_{-n}^*(p)$, and the formula follows from Theorem~\ref{THM:PowerSumporder} by replacing $n$ with $-n$ and setting $d=q=1$.
\end{proof}

%In particular, for prime $p\ge3$,
%\[
%1+\frac{1}{2} +\dotsb+\frac{1}{p-1} \equiv\ 0\ \,\pmod{p},        
%\]
%and for $p\ge5$,
%\[
%%$S_{-1}(p-1)   \equiv 0 \pmod{p^2}$
%1+\frac{1}{2^2}+\frac{1}{3^2} +\dotsb+\frac{1}{(p-1)^2} \equiv\ 0\ \,\pmod{p}.               
%\]
Taking $n=1$, the congruence $S_{-1}(p-1) \equiv 0$ actually holds modulo $p^2$, if $p\ge5$, by {\em Wolstenholme's theorem} \cite{Wo,Me}.

The following theorem provides us with additional information about the divisibility of power sums.

\begin{proposition} \label{powerSumDivisibilitySharpened} Given integers $p,q\ge1, n \geq 0$, and $d \geq c \geq 0$, set $\delta=d-c$. Then
\[
S_n(p^dq)=p^{\delta}qS_n(p^c)+\sum_{k=1}^n \binom{n}{k} p^{ck} ( S_k(p^{\delta}q)-(p^{\delta}q)^k) S_{n-k}(p^c).
\]
\end{proposition}

\begin{proof}
We have
\begin{align*}
S_n(p^dq)&=\sum_{j=0}^{p^{\delta}q-1} \sum_{i=1}^{p^c} (jp^c+i)^n= \sum_{j=0}^{p^{\delta}q-1} \sum_{i=1}^{p^c} \left(i^n+\sum_{k=1}^{n}\binom{n}{k} (j p^c)^k i^{n-k}\right)\\
&= \sum_{j=0}^{p^{\delta}q-1} \sum_{i=1}^{p^c} i^{n}+\sum_{k=1}^{n}\binom{n}{k} p^{ck}\sum_{j=0}^{p^{\delta}q-1}j^k \sum_{i=1}^{p^c}   i^{n-k}.
\end{align*}
Using $\sum_{j=0}^{p^{\delta}q-1}j^k =  S_k(p^{\delta}q)-(p^{\delta }q)^k$, the desired formula follows.
\end{proof}

\begin{corollary} \label{Snp2 mod p3} For any prime $p \geq 5$ and integer $n \geq 0$, the following congruence holds:
\[
S_n(p^2) \equiv p S_n(p)+pnS_{n-1}(p)(S_1(p)-p) \pmod{p^3}.
\]
\end{corollary}

\begin{proof}
If $n=0$ or $1$, it is easy to verify the congruence. Now, assume that $n \geq 2$ and set $d=2$ and $c=1$ in Theorem \ref{powerSumDivisibilitySharpened}. Then 
\begin{align*}
S_n(p^2)&=p S_n(p)+\sum_{k=1}^n \binom{n}{k} p^{k} ( S_k(p)-p^{ k}) S_{n-k}(p)\\
% &\equiv p S_n(p)+\sum_{k=1}^2 \binom{n}{k} p^{k} ( S_k(p)-p^{ k}) S_{n-k}(p)\\
&\equiv p S_n{p}+np(S_1(p)-p) S_{n-1}(p)+ \binom{n}{2}p^2 (S_2(p)-p^2) S_{n-2}(p)  \pmod{p^3}.
\end{align*}
Since $p \geq 5$ implies that $S_2(p)$ is divisible by $p$, the proof is complete.
\end{proof}

Corollary \ref{Snp2 mod p3} fails with $p=3$. Indeed, take $n=1$. Then $S_n(p^2)=S_1(9)=45$, whereas $pS_n(p)+p n S_{n-1}(p)(S_1(p)-p) $ equals
\begin{align*}
 3S_1(3)+3 S_{2}(3)(S_1(3)-3)=18+3 \cdot 14(6-3)=144\not\equiv 45 \pmod{3^3}.
\end{align*}

Recall that {\em Pascal's identity} is
\begin{equation}
\sum_{k=0}^{n-1} \binom{n}{k} S_{k}(a)=(a+1)^{n}-1, \label{EQ:Pascal}
\end{equation}
valid for $a \ge0$ and $n\ge1$ (see, e.g., \cite{MS1}). Here is an analog for even exponents.

\begin{theorem}[A Pascal identity for even exponents] \label{THM:evenPascal} For any integer $a \ge0$ and even $n\ge2$,
\[
\sum_{k=0}^{(n-2)/2} \binom{n}{2k} S_{2k}(a)=\frac12\left((a+1)^{n}-(a^{n}+1)\right).
\]
\end{theorem}

\begin{proof}
Since $n$ is even, the Binomial Theorem gives
\begin{align*}
S_n(a+1)+S_n(a-1)-1&=\sum_{j=1}^{a} \left((1+j)^n+(1-j)^n\right)= \sum_{j=1}^{a}\sum_{k=0}^{n} \binom{n}{k} j^k(1+(-1)^k)\\
&= 2 \sum_{j=1}^{a}\sum_{k=0}^{n/2}  \binom{n}{2k}j^{2k} = 2 \sum_{k=0}^{n/2}  \binom{n}{2k}\sum_{j=1}^{a}j^{2k}.
\end{align*}
Using $S_n(m)=\sum_{j=1}^m j^n=S_n(m-1)+m^n$, we can write this as
\[
2S_n(a)+(a+1)^n-a^n-1=2 \sum_{k=0}^{n/2} \binom{n}{2k} S_{2k}(a).
\]
As $n\ge2$, subtracting $2S_n(a)$ from both sides and then dividing by~$2$ yields the desired formula.
\end{proof}

For an application of Pascal's identity to Bernoulli numbers, see Section~\ref{Bernoulli numbers}.

\begin{theorem} \label{CongruenceOfPowerSums}
Let $p$ be an odd prime and let $m$ and $n$ be positive integers.\\
{\rm(i).} For some integer $d\ge1$, we can write
$$m=qp^d + r\, \frac{p^d-1}{p-1} =qp^d+rp^{d-1}+rp^{d-2}+\dotsb+rp^0,$$
where $r \in \{0,1,\dots,p-1\}$ and $0\le q \not \equiv r \equiv m\pmod{p}$.\\
{\rm(ii).}  In case $m \equiv 0 \pmod{p}$, we have
\[
S_n(m) \equiv \begin{cases}-p^{d-1}q \pmod{p^d} \ \quad \text{if }p-1\mid n,\\
                                       \  0 \quad\quad\, \pmod{p^d}\quad \text{ if }p-1\nmid n.
                                         \end{cases}
\]
{\rm(iii).} In case $m \equiv -1 \pmod{p}$, we have
\[
S_n(m) \equiv \begin{cases}-p^{d-1}(q+1) \pmod{p^d} \quad\text{ if }p-1\mid n,\\
                                         \ 0\qquad\quad\quad\ \ \pmod{p^d}\quad \text{ if }p-1\nmid n.
                                         \end{cases}
\]
{\rm(iv).} In case $m \equiv \frac{p-1}{2} \pmod{p}$, we have
\[
S_{n}(m) \equiv \begin{cases}  -p^{d-1}\!\left(q+\dfrac{1}{2}\right) \pmod{p^d}\quad  \text{ if }p-1\mid n,\\
                                        \ 0 \qquad\qquad\qquad\!\pmod{p^d} \quad\text{ if }p-1\nmid n \text{ and $n$ is even}.
                                         \end{cases}
\]
\end{theorem}

\begin{proof}
Since $m>0$, we can write it in base~$p$ as
$m={a_ka_{k-1}\dots a_{d} a_{d-1}\dots a_{1}a_{0}}_p$ with a leading zero $a_k=0$,
all $a_i \in \{0,1,\dots,p-1\}$, and $r:=a_0=a_1=\dots=a_{d-1}\neq a_d$, where $d\ge1$. Then $m=qp^d+ r\, \frac{p^d-1}{p-1}$, where $0\le q=\sum_{i=0}^{k-d}a_{d+i}p^i  \equiv a_d\not \equiv r \equiv m\pmod{p}$, proving (i).

%q=a_d+a_{d+1}p+a_{d+2}p^2+\dotsb +a_kp^{k-d}
If $m \equiv 0 \pmod{p}$, then $r=0$. Hence $m=p^dq$, and Corollary \ref{PowerSumporder} implies (ii).

Reducing binomials of the form $(q p^d+j)^n$ modulo $p^d$ shows that
\[
S_{n}(m)=S_{n}\!\left(q p^d+ r\,\frac{p^d-1}{p-1}\right) \equiv S_{n}(q p^d)+S_{n}\!\left(r\,\frac{p^d-1}{p-1}\right) \pmod{p^{d}},
\]
and Corollary \ref{PowerSumporder} computes the term $S_{n}(qp^d)$ modulo $p^d$. It remains to compute the last term modulo $p^d$ in case $m \equiv -1$ or $\frac{p-1}{2} \pmod{p}$.

If $m \equiv -1 \pmod{p}$, then $r= p-1$. Now,
\begin{align*}
S_{n}\!\left( r\,\frac{p^d-1}{p-1}\right)= S_{n}(p^d-1)= S_{n}(p^d) -p^{dn}\equiv S_{n}(p^d) \pmod{p^d},
\end{align*}
and another application of Corollary \ref{PowerSumporder} yields (iii).

Finally, if $m \equiv \frac{p-1}{2} \pmod{p}$, then $r= \frac{p-1}{2}$ and
\begin{align*}
S_{n}\!\left( r\,\frac{p^d-1}{p-1}\right)=S_{n}\!\left( \frac{p^d-1}{2}\right).
\end{align*}
To compute the latter modulo $p^d$ when $n$ is even, we write
\[
S_n(p^d-1) = \sum_{k=1}^{(p^d-1)/2} \!\left(k^n + (p^d-k)^n\right) \equiv 2S_{n}\!\left( \frac{p^d-1}{2}\right) \pmod{p^d}.
\]
%since $S_n(p^d-1) = S_n(p^d)-p^d$. 
Since $S_{n}(p^d-1)\equiv S_{n}(p^d) \pmod{p^d},$ we get
\[
 \text{$n$ even}\implies S_{n}\!\left(\frac{p^{d}-1}{2}\right) \equiv  \frac{1}{2} S_n(p^d) \pmod{p^d},\\
\]
and a final application of Corollary \ref{PowerSumporder} gives (iv).
\end{proof}

%For $r\ge2$, let $v_r(q)$ denote the exponent of the highest power of $r$ that divides~$q$. Let $x=r^{v_r(x)} k$ and $y=r^{v_r(y)} l$. The function $v_r(\cdot)$ is additive for any $r$ in the sense that $v_r(x\cdot y)=v_r(x)+v_r(y)$ if $k$ and $l$ are relatively prime.

\begin{definition}\label{DEF:order}
 For $q\in\Z$ and prime~$p$, the \emph{$p$-adic order of $q$} is the exponent $ v_p(q)$ of the highest power of $p$ that divides~$q$:
 $$v_p :\Z\to\N\cup\{0,\infty\}, \qquad v_p(q):=\sup_{p^d\mid q}d.$$
 \end{definition}

The function $v_p(\cdot)$ is totally additive: $v_p(x\cdot y)=v_p(x)+v_p(y)$ for any $x$ and $y$. Note that $v_p(q)\in\N\cup\{0\}$ for $q\neq0$, and $v_p(0)=\infty$.

For the next result, we will find it useful to write a positive integer $m$ in a certain nice form which allows us to determine the least $d$ for which $S_{n}(m) \pmod{p^d} $ is not zero for $n$ divisible by $p-1$. More generally, we let $m$ lie in the {\it $p$-adic integers} $\Z_p$ and note that $v_p$ can be defined on $\Z_p$ by considering the digits of the base $p$ expansion.

\begin{definition}
Define a map $V_p: \mathbb{Z}_p \rightarrow \N\cup \{0,\infty\}$ by
\[
V_p(m):=v_p(m-\lfloor\frac{m}{p}\rfloor)+1.
\]
\end{definition}

This function can be interpreted as follows: $V_p(m)$ counts the number of equal $p$-digits at the end of the base $p$ expansion of $m \in {\Z}_p$.

\begin{lemma} \label{lemma: lastDigits} Write $m \in {\Z}_p$ in base~$p$ as
\[
m={\dots a_{k}  \dots a_1 a_0}_p=\sum_{i=0}^{\infty} a_i p^i,
\]
with $a_i \in \{0,1,\ldots,p-1\}$ for each $i$. Let 
\[
h=\sup\{i \in \N\cup \{0\}: a_i=a_j  \ \forall \ 0 \leq j \leq i\}.
\]
Then $V_p(m)=h+1$.
\end{lemma}

\begin{proof} Indeed,
\[
m-\lfloor\frac{m}{p}\rfloor=\sum_{i=0}^{\infty} a_{i} p^i-\sum_{i=0}^{\infty} a_{i+1} p^i.
\]
If $h=\infty$ then the result follows. Assume then that $h$ is finite. For each of the indices $i=1,2,\ldots,h-1$, we have $a_i=a_{i+1}$. For the index $i=h$, by assumption $a_{h} \neq a_{h+1}$. Therefore $v_p(m-\lfloor\frac{m}{p}\rfloor)=h$. 
\end{proof}

A few comments regarding $V_p$ are in order. From Lemma \ref{lemma: lastDigits}, we see that $V_p(m) = \infty$ exactly when all base $p$ digits of $m$ are the same. The values of $m\in\Z_p$ for which this occurs are
$$m=-\frac{r}{p-1}={\dots r  rr}_p=\sum_{i=0}^{\infty} r p^i
$$
for $r \in \{0,1,\ldots,p-1\}$. In particular, this is the case for $m=-1,0,$ and $-\frac{1}{2}$ when $p$ is odd.

Let $V_p(m)=d$. Then, as in Theorem \ref{CongruenceOfPowerSums}, we may write $m=q p^d +a_0 \sum_{k=0}^{d-1} p^{k}$ with $0\le q \not \equiv a_0 \pmod{p}$.  

\begin{remark} \label{REM:V_p}
{\em If $m \equiv -1 \pmod{p}$, then the equalities} $V_p(m)=V_p(m+1)=v_p(m+1)$ {\em hold}. Indeed, write $m$ in base~$p$ as
\[
m= \dots a_h (p-1) (p-1) \dots (p-1)_p,
\]
with $a_h \neq p-1$ so that $V_p(m)=h$. Notice that $a_h \neq p-1$ implies $v_p(m+1)=h$ because $m+1= \dots a_{h+1}(a_h+1) 00 \dots 0_p,$ since $a_h<p-1$. Thus $V_p(m)=V_p(m+1)=v_p(m+1)$.
\end{remark}

\begin{theorem} \label{porderOfPowerSums}
Let $p$ be an odd prime and let $m$ be a positive integer.\\
{\rm(i).} In case $m\equiv0$ or $-1\pmod{p}$, we have
\[
v_p(S_n(m)) \begin{cases} =v_p(S_{p-1}(m))=V_p(m)-1\quad \text{ if }p-1\mid n,\\
                                                 \geq V_p(m)\qquad\qquad\qquad\qquad\, \quad\text{ if }p-1\nmid n.
\end{cases}
\]
{\rm(ii).} In case $m\equiv \frac{p-1}{2} \pmod{p}$, we have
\[
v_p(S_n(m)) \begin{cases} =v_p(S_{p-1}(m))=V_p(m)-1\quad \text{ if $n$ is even},\\
                                                 \geq V_p(m)\qquad\qquad\qquad\qquad\, \quad\text{ if $n$ is odd.}
\end{cases}
\]
\end{theorem}

\begin{proof}
This follows immediately from Theorem \ref{CongruenceOfPowerSums}.
\end{proof}

As an example, take $p=3$ and $m=1222_3$ in base $3$. In particular, there are three copies of $2$ at the end, so we know that $V_3(m)=3$. By Theorem~\ref{porderOfPowerSums}, for any even~$n$,
\[
v_3(S_n(m))=v_3(S_2(m))=V_3(m)-1=2.
\]
As $m=53$, this agrees with the fact that $S_2(53)=53\cdot54(2\cdot53+1)/6=51039=3^2 \cdot 53 \cdot 107$.

We note that Theorem \ref{porderOfPowerSums} is tight. Indeed, take $p=5$ and $n=8$, so that $p-1\mid n$. Besides $m\equiv0, (p-1)/2,p-1\pmod{p}$, consider the remaining two congruence classes, namely $m\equiv1,3 \pmod{5}$. First, take $m=6 \equiv 1 \pmod{5}$. We then have $S_4(6)=2275 \equiv 0 \pmod{25}$, whereas $S_8(6)=2142595 \equiv 20 \pmod{25}$. Now take $m=18 \equiv 3 \pmod{5}$. Then $S_4(18)=432345 \equiv 20 \pmod{25}$, whereas $S_8(18)=27957167625 \equiv 0 \pmod{25}$. Thus in both cases $v_p(S_n(m)) \neq v_p(S_{p-1}(m)).$

As an application, we obtain a simple proof of the following classical result.

\begin{corollary} \label{FactorsOfPowerSumPolynomial}
For even $n$, the polynomial in $\Q[x]$ interpolating $S_n(x)$ is divisible by the product $x(x+1)(2x+1)$.
\end{corollary}

\begin{proof}
Fix an odd prime $p$. First, consider the sequence $x_i=p^i$, for $i=1,2,\ldots$. We have $v_p(x_i)=i$, so that $x_i \rightarrow 0$ $p$-adically. On the other hand, $v_p(S_n(x_i))\geq V_p(x_i)-1=i-1$ by Theorem \ref{porderOfPowerSums}. Therefore $S_n(x_i) \rightarrow 0$ $p$-adically. By continuity, $x=0$ is a root of $S_n(x)$.

Similarly, consider the sequence $x_i=\sum_{j=0}^{i} (p-1)p^j$, for $i=1,2,\ldots$. This sequence converges $p$-adically to $-1$. Theorem \ref{porderOfPowerSums} gives $v_p(S_n(x_i)) \geq V_p(x_i)-1=i-1$. Therefore, $x=-1$ is a root of $S_n(x)$.

Finally, the sequence $x_i=\sum_{j=0}^{i} \frac{p-1}{2}p^j$, which converges $p$-adically to $-1/2$, shows that $x=-1/2$ is a root of $S_n(x)$.
\end{proof}

The next result gives two special cases of Theorem \ref{porderOfPowerSums}.

\begin{corollary} \label{COR: p=3}
Let $m$ and $n$ be positive integers.\\
{\rm(i).} The $3$-adic order of $S_{2n}(m)$ equals
\[
v_3(S_{2n}(m))=v_3(m(m+1)(2m+1)/3) = V_3(m)-1.
\]  
{\rm(ii).} If $m\equiv 0,2,$ or $4 \pmod{5}$, then the $5$-adic order of $S_{4n}(m)$ equals
\[
v_5(S_{4n}(m))=v_5(m (m+1) (2 m+1) (3 m^2 + 3 m -1)/5) = V_5(m)-1.
\]
\end{corollary}

\begin{proof}
%{\rm(i).} Take $p=3$ in Corollary \ref{porderOfPowerSums}, and use $S_2(m)=m (m+1) (2 m+1)/6$.\\
%{\rm(ii).} Take $p=5$ in Corollary \ref{porderOfPowerSums}, and use $S_4(m)=m (m+1) (2 m+1) (3 m^2 + 3 m -1)/30$.
Take $p=3$ and $5$ in Theorem \ref{porderOfPowerSums}, and use the formulas $S_2(m)=m (m+1) (2 m+1)/6$ and
$S_4(m)=m (m+1) (2 m+1) (3 m^2 + 3 m -1)/30$, respectively.
\end{proof}

We recall an analogous result for the prime~$2$. (The result is not used in this paper.)

\begin{theorem}[MacMillan and Sondow \cite{MS2}] \label{THM: Sigma(2^d)}
For any positive integers $m$ and $n$, the $2$-adic order of $S_n(m)$ equals
\begin{equation*} %\label{EQ: main}
\begin{aligned}
	v_2\!\left(S_n(m)\right) =
		\begin{cases}
			\, v_2(m(m+1)/2) &\text{if $n=1$ or $n$ is even,}  \\[0.4em]
			\, 2v_2(m(m+1)/2) &\text{if $n\ge3$ is odd.}
		\end{cases}
\end{aligned} 
\end{equation*} 
\end{theorem}
%\end{theorem*}

As an application of our results to the the Erd\H{o}s--Moser equation, we have the following theorem. Part (i) is due to Moree \cite[Proposition~9]{Moree}.

\begin{theorem} \label{thm: EMresults}
Let $p$ be an odd prime. \\
{\rm(i).} In the generalized Erd\H{o}s--Moser equation, if $p\mid m+1$, then $p-1 \nmid n$.\\
{\rm(ii).} In the Erd\H{o}s--Moser equation, if $p\mid m$, then $p-1 \mid n$ and $p^2\mid m+p$. Also, if $p\mid m- \frac{p-1}{2}$, then $p-1 \mid n$ and $m \equiv   -(p+\frac{1}{2}) \pmod{p^2}$.
\end{theorem}

\begin{proof}
(i). Assume that $m \equiv -1 \pmod{p}$. Then by Remark \ref{REM:V_p} we have $V_p(m)=v_p(m+1)$. If $p-1 \mid n$, then using Theorem \ref{porderOfPowerSums} and applying $v_p$ to both sides of equation \eqref{EQ:GEME} gives
\[
V_p(m)-1=v_p(S_n(m))=v_p(a)+nv_p(m+1)=v_p(a)+nV_p(m),
\]
contradicting $v_p\ge0$ and $V_p\ge0$. Therefore $p-1 \nmid n$. \\
(ii). If $p \mid m$, write $m=p^dq$, with $d>0$ and $p\nmid q$. Reducing both sides of \eqref{EQ:EME} modulo $p^d$, we deduce that
$S_n(m) \equiv 1 \pmod{p^d}.$
Hence, by Theorem \ref{CongruenceOfPowerSums}, we must have $p-1 \mid n$ and 
\[
S_n(m)  \equiv -p^{d-1} q \pmod{p^d}.
\]
Thus $-\frac{m}{p} = -p^{d-1} q \equiv 1 \pmod{p^d}$. Since $d\ge1$, this implies $m \equiv -p \pmod{p^2}$.

If $m \equiv \frac{p-1}{2} \pmod{p}$, write $m=a_d p^d+\frac{p^d-1}{2}$. Reducing both sides of \eqref{EQ:EME} modulo $p^d$, we see that 
\[
S_n(m) \equiv \left(\frac{p^d+1}{2}\right)^n \pmod{p^d}.
\]
By Theorem \ref{CongruenceOfPowerSums}, we see that $p-1 \mid n$ and 
\[
-p^{d-1}(a_d+2^{-1}) \equiv  \left(\frac{p^d+1}{2}\right)^n \equiv (2^{-1})^n \pmod{p^d}.
\]
Hence $d=1$. Using the fact that the multiplicative order of any element of $(\Z/p\Z)^*$ divides $p-1$, we obtain $a_d \equiv -1-2^{-1} \pmod{p}$. Therefore $m \equiv -p-2^{-1} \pmod{p^2}$.
\end{proof}

%\begin{corollary} \label{divisBy3ofEM} Assume that $n$ is even. \\
%{\rm(i).} The generalized Erd\H{o}s--Moser equation has no solution $(m,n)$ with $m \equiv -1 \pmod{3}$. \\
%{\rm(ii).} The Erd\H{o}s--Moser equation has no solution $(m,n)$ with $m \equiv 0$ or $4 \pmod{9}$. 
%\end{corollary}

\begin{theorem} \label{cor:EMCongruences}
{\rm (i).} Any non-trivial solution of the generalized Erd\H{o}s--Moser equation must have $m \equiv 0$ or $4 \pmod{6}$. Furthermore, if $m\equiv4\pmod{5}$, then $n\equiv2\pmod{4}$.\\
{\rm (ii).} Any non-trivial solution of the Erd\H{o}s--Moser equation must have $m \equiv6$ or $10\pmod{18}$. 
\end{theorem}

\begin{proof}
(i). By \cite{Moree, Moser} (see also \cite{MS2}), any non-trivial solution of \eqref{EQ:GEME} has $m\equiv n\equiv0\pmod{2}$. Since $n$ is even, Theorem~\ref{thm: EMresults} part (i) implies $m \not\equiv2\pmod{3}$. Hence $m \equiv 0$ or $4 \pmod{6}$, proving the first part of (i). The second part follows from Corollary~\ref{COR: p=3} part (ii).\\
(ii). Since $n$ is even, we can apply Corollary~\ref{COR: p=3} part (i) to equation \eqref{EQ:EME}, yielding
\[
v_3(m(m+1)(2m+1))-1=n  v_3(m+1),
\]
that is,
\begin{equation*} \label{EQ:v_3}
v_3(m) + v_3(2m+1)=1+(n-1)  v_3(m+1).
\end{equation*}
It follows that $m \equiv1,3,6,$ or $7\pmod{9}$.

According to \cite[Equations 6, 10, 12, 13]{MoreeTopHat}, in any solution $(m,n)$ of the Erd\H{o}s--Moser equation, $m,\frac{m+2}{2},2m+1$, and $2m+3$ are all square-free. Also, Moree \cite[Theorem 1]{MoreeOnMoser}, whose $m$ is our $m+1$, showed that our $m \equiv 0 \pmod{2}$. The condition that $2m+3$ is square-free eliminates the case $m \equiv 3 \pmod{9}$. In the case $m \equiv 7 \pmod{9}$, the Chinese Remainder Theorem would imply $m \equiv 34 \pmod{72}$, contradicting the square-freeness of
\[
\frac{m+2}{2} \equiv 18 \pmod{36}.
\]
Therefore $m \equiv 1$ or $6 \pmod{9}$. Since $m$ is even, it follows that $m \equiv6$ or $10\pmod{18}$.
\end{proof}

%%%
\section{Egyptian Fraction Equations}

Fix a positive integer $n$. The congruence
\begin{align} \label{EQ:coolNums}
\sum_{p \mid n} \frac{1}{p}+\frac{d}{n} \equiv 1 \pmod{1}
\end{align}
is equivalent to the congruence
\begin{align} \label{EQ:d}
d \equiv - \sum_{p \mid n} \frac{n}{p} \pmod{n}.
\end{align}
In particular, there are always integer solutions~$d$.

\begin{definition}  \label{DEF:d(n)}
We denote one solution of \eqref{EQ:coolNums} by
\begin{align} \label{EQ:d(n)}
d(n):= - \sum_{p \mid n} \frac{n}{p}.
\end{align}
If $n$ is composite and $d(n)\equiv -1 \pmod{n}$, then $n$ is called a {\em Giuga number}.
\end{definition}

In other words, {\em a Giuga number is a composite number $n$ satisfying the Egyptian fraction condition}
%{\bf I think it should be in $\Z$}
\begin{align*} %\label{EQ:Giuga}
\sum_{p \mid n} \frac{1}{p}-\frac{1}{n} \in \N.
\end{align*}
All known Giuga numbers $n$ in fact satisfy the Egyptian fraction {\em equation}
\begin{align*} %\label{EQ:strongGiuga}
\sum_{p \mid n} \frac{1}{p}-\frac{1}{n} = 1,
\end{align*}
which holds if and only if $d(n)=-1-n$. In that case, we call $n$ a {\em strong Giuga number}.
 %It is easy to see that such $n$ are composite. Therefore, every strong Giuga number is a fortiori a Giuga number. 
 The first few (strong) Giuga numbers are \cite{BBBG}, \cite{Me2}, \cite[Sequence A007850]{Sloane}
$$n=30, 858, 1722, 66198, 2214408306, 24423128562, 432749205173838,\dotsc.$$

\begin{definition}
If $n>1$ and $d(n)= 1-n$, then $n$ is called a {\em primary pseudoperfect number}.
\end{definition}

Equivalently, Butske, Jaje, and Mayernik \cite{Butske} {\em define a primary pseudoperfect number to be a solution $n>1$ to the Egyptian fraction equation}
\begin{align*} %\label{EQ:PPP}
\sum_{p \mid n} \frac{1}{p}+\frac{1}{n} =1.
\end{align*}

It follows from Definition \ref{DEF:d(n)} that {\em if $d(n)\equiv \pm1 \pmod{n}$, then $n$ is square-free}. In particular, {\em all Giuga and primary pseudoperfect numbers are square-free}.

The primary pseudoperfect numbers with $k\le8$ (distinct) prime factors are \cite[Table~1]{Butske}, \cite[Sequence A054377]{Sloane}
$$n_k=2, 6, 42, 1806, 47058, 2214502422, 52495396602, 8490421583559688410706771261086.$$
Each $n_k$ has exactly $k$ (distinct) prime factors, $k=1,2,3,4,5,6,7,8$. Moreover, the $n_k$ are the only known solutions to the congruence $d(n)\equiv 1 \pmod{n}$.

In some cases the next result can be used to generate new Giuga and primary pseudoperfect numbers from given ones. Part (i) is from \cite{wiki} and part (iii) is a special case of Brenton and Hill \cite[Proposition 12]{BH} (see also \cite[Lemma~4.1]{Butske}).

\begin{theorem} \label{PROP:pppGiuga}
{\rm(i).} Assume $n+1$ is an odd prime. Then $n$ is a primary pseudoperfect number if and only if $n(n+1)$~is also a primary pseudoperfect number.\\
{\rm(ii).} Assume $n-1$ is a prime. Then $n$~is a primary pseudoperfect number if and only if $n(n-1)$ is a strong Giuga number.\\
{\rm(iii).} Assume $n^2+1=FG$, where $n+F$ and $n+G$ are prime. Then $n$ is a primary pseudoperfect number if and only if $n(n+F)(n+G)$ is also a primary pseudoperfect number.\\
{\rm(iv).} Assume $n^2-1=FG$, where $n+F$ and $n+G$ are prime. Then $n$ is a primary pseudoperfect number if and only if $n(n+F)(n+G)$ is a strong Giuga number.
\end{theorem}

\begin{proof}
In the proof of (i), (ii), take all $\pm$ signs to be $+$, or all to be $-$, and likewise in the proof of (iii), (iv).\\
(i), (ii). We can write
\begin{align*}
\displaystyle \sum_{p \mid n} \frac{1}{p}+\frac{1}{n} = \sum_{p \mid n} \frac{1}{p}+\frac{1}{n\pm1}+\left(\frac{1}{n}-\frac{1}{n\pm1}\right)= \sum_{p \mid n(n\pm1)} \frac{1}{p} \pm \frac{1}{n(n\pm1)},
 \end{align*}
as $n\pm1$ is prime. This implies (i) and (ii).\\
(iii), (iv). Since $n^2\pm1=f^2$ has no solutions in positive integers, the primes $n+F$ and $n+G$ are distinct. Setting $M:=n(n+F)(n+G)$, we therefore have
\begin{align*}
\sum_{p \mid M} \frac{1}{p} \pm\frac{1}{M}&=\sum_{p \mid n}\frac{1}{p} + \frac{1}{n+F}+ \frac{1}{n+G}\pm \frac{1}{M}=\sum_{p \mid n}\frac{1}{p} + \frac{n(n+F)+n(n+G)\pm1}{M}\\
&=\sum_{p \mid n} \frac{1}{p} + \frac{1}{n},
\end{align*}
because $n^2\pm1=FG$ implies $n(n+F)+n(n+G)\pm1=(n+F)(n+G)$. This proves (iii) and (iv).
\end{proof}

\begin{example} \label{EX:N+/-1}
For examples of (i), let $n$ be one of the four primary pseudoperfect numbers$$2,\quad 6=2\cdot3,\quad 42=2\cdot3\cdot7,\quad 47058=2\cdot3\cdot11\cdot23\cdot31.$$Then the primes $n+1 =3, 7, 43,47059$ yield the primary pseudoperfect numbers
$$n(n+1) = 6, 42, 1806, 2214502422$$

For (ii), if $n=6, 42,$ or $ 47058$, then $n-1=5,41,$ or $47057$ is prime, and the products
$$n(n-1) = 30, 1722, 2214408306$$
are strong Giuga numbers.
\end{example}

Notice here the three pairs of twin primes$$(n-1,n+1)=(5,7), (41,43), (47057,47059).$$Is this more than just a coincidence? In other words:

\begin{question}
Let $n>2$ be a primary pseudoperfect number. Is $n-1$ prime if and only if $n+1$ is prime? Equivalently (by Theorem~\ref{PROP:pppGiuga}), is $n(n-1)$ a strong Giuga number if and only if $n(n+1)$ is a primary pseudoperfect number?
\end{question}

\begin{example}
The only known example of Theorem \ref{PROP:pppGiuga} part (iii) begins with the primary pseudoperfect number
$$n_6= 2214502422=2\cdot3\cdot11\cdot23\cdot31\cdot47059.$$
Factoring
$$n_6^2+1=4904020977043866085=2839805 \cdot1726886521097=:F^+\cdot G^+$$
%yields the primes $n+F=2217342227$ and $n+G=1729101023519$. Then
leads to the primes $n_6+F^+$ and $n_6+G^+$ and then to the largest known primary pseudoperfect number
\begin{align*}
n_8=n_6(n_6+F^+)(n_6+G^+)&=2\cdot3\cdot11\cdot23\cdot31\cdot47059\cdot2217342227\cdot1729101023519\\
&=8490421583559688410706771261086.
%\\&=2\cdot3\cdot11\cdot23\cdot31\cdot47059\cdot2217342227\cdot1729101023519
\end{align*}
%is also a primary pseudoperfect number, as the reader may verify from its prime factorization.
%\begin{align*}
%M=2\cdot3\cdot11\cdot23\cdot31\cdot47059\cdot 2214502422\cdot2217342227\cdot1729101023519.
%\end{align*}

The number $n_6$ also provides an example of (iv). Namely, the factorization
$$n_6^2-1=4904020977043866083=45193927\cdot 108510618629=:F^-\cdot G^-$$
yields the primes $n_6+F^-$ and $n_6+G^-$ and hence the strong Giuga number
%$$ n(n+F^+)(n+G^+)= 2259696349 \cdot 110725121051=554079914617070801288578559178.$$
\begin{align*}
n_6(n_6+F^-)(n_6+G^-)&=2\cdot3\cdot11\cdot23\cdot31\cdot47059 \cdot 2259696349 \cdot 110725121051\\
&=554079914617070801288578559178.
\end{align*}

Another example of (iv) begins with $n_8$ and ends with the largest known (strong) Giuga number
{\Tiny
\begin{align*}
&2\cdot3\cdot11\cdot23\cdot31\cdot47059\cdot2217342227\cdot1729101023519\cdot58254480569119734123\cdot8491659218261819498490029296021\\
&=4200017949707747062038711509670656632404195753751630609228764416142557211582098432545190323474818
541298976556403,
\end{align*}
}discovered by R. Girgensohn \cite{BBBG}.
%\begin{align*}
%M(M+F')(M+G')&=M\cdot8491659218261819498490029296021\cdot5825448056911973412354\\
%                          &=4200017949707747062038711509670656632404195753751630609\diagdown\\ 
%                          &\quad\ 228764416142557211582098432545190323474818.
%\end{align*}
\end{example}

\begin{proposition} \label{coolEquationEquivalentConditions}
An ordered pair $(n,d)$ is a solution to the congruence \eqref{EQ:coolNums} if and only if
\begin{equation}
p\mid n \implies d \equiv -\frac{n}{p} \pmod{{p}^{v_p(n)}}. \label{EQ:coolEEC}
\end{equation}
In that case, let $p$ be a prime factor of $n$ and $e\in \N$. Then $p^e$ divides~$n$ if and only if $p^{e-1}$ divides~$d$. In particular, $n$ is square-free if and only if $n$ and $d$ are coprime.
\end{proposition}

\begin{proof}
If $(n,d)$ is a solution, then \eqref{EQ:d} reduced modulo $p^{v_p(n)}$ implies \eqref{EQ:coolEEC}. The converse follows from the Chinese Remainder Theorem, and we infer the proposition.
\end{proof}

%\begin{corollary} \label{CoolDivisibilityCorollary}
%Let $(n,d)$ be a solution to \eqref{EQ:coolNums} and let $p$ be a prime factor of $n$. Then $p^e$ divides~$n$ if and only if $p^{e-1}$ divides $d$. In particular, $n$ is square-free if and only if $n$ and $d$ are coprime.
%\end{corollary}

The next theorem gives three properties of the function $n \mapsto d(n)$. The first is a power rule. The second shows that the function $n \mapsto d(n)$ satisfies Leibnitz's product rule, but only on coprime integers; in other words, it is ``Leibnitzian,'' but not ``totally Leibnitzian.'' The third is an analog of the quotient rule.

\begin{theorem} \label{thm: LeibnitzRule}
{\rm(i).} For $k,n \in \N$, we have $d(n^k)=n^{k-1} d(n)$.\\
{\rm(ii).} Given $M,n \in \N$, denote their greatest common divisor by $G:=\gcd(M,n)$ and their least common multiple by $L:=\lcm(M,n)$. Then 
\[
d(Mn)=Md(n)+n d(M)-L d(G).
\]
In particular,
\[
\gcd(M,n)=1 \implies d(Mn)=Md(n)+n d(M).
\]
{\rm(iii).} Let $a$ and $b$ be positive integers with $b \mid a$. Set $\gamma :=\gcd(b,a/b)$. Then
\[
d\left(\frac{a}{b}\right)=\frac{bd(a) -a d(b)}{b^2}+\frac{a/b}{\gamma}d(\gamma ).
\]
In particular, when $\gamma =1$ we have the standard quotient rule.
\end{theorem}

\begin{proof}
{\rm(i).} By Definition \ref{DEF:d(n)},
\[
d(n^k)=-\sum_{p\mid n^k} \frac{n^k}{p}=-n^{k-1}\sum_{p\mid n} \frac{n}{p} =n^{k-1} d(n).
\]
{\rm(ii).} Since $G=\gcd(M,n)$,
\[
\sum_{p\mid Mn} \frac{1}{p} =\sum_{p\mid M} \frac{1}{p} + \sum_{p\mid n} \frac{1}{p}-\sum_{p\mid G} \frac{1}{p}.
\]
Multiplying through by $-Mn$, we write the result as
\begin{align*}
d(Mn)=-\sum_{p\mid Mn} \frac{Mn}{p} &=-n\sum_{p\mid M} \frac{M}{p}-M \sum_{p\mid n} \frac{n}{p}+\frac{Mn}{G}  \sum_{p\mid G}\frac{G}{p}.
%\\&=n d(M)+Md(n)-L d(G)
\end{align*}
Since $L=Mn/G$, the first conclusion follows. If $G=1$, then $\sum_{p\mid G} \frac{1}{p}=0$, and we get the product rule.\\
{\rm(iii).} By part (ii),
\[
d(a)=d\left(b \frac{a}{b}\right)=\frac{a}{b} d(b)+b d\left(\frac{a}{b}\right)-\frac{a}{\gamma} d(\gamma).
\]
Dividing by $b$ and solving for $d(\frac{a}{b})$ yields (iii).
\end{proof}

For a prime $p$, Definition \ref{DEF:d(n)} gives
\begin{align} \label{EQ:d(p)}
d(p)=- \frac{p}{p}=-1
\end{align}
On the other hand, the \emph{arithmetic derivative} \cite{Barbeau, Churchill, UfnarovskiAhlander} of $p$ is defined as $p'=1$, and that of a product $ab$ is defined as $(ab)'=ab'+ba'$. (Also, $0':=1':=0.$) Thus, for square-free $n>1$, both $d(n)$ and the arithmetic derivative $n'$ can be calculated by applying Leibnitz's product rule to the prime factorization of~$n$. Therefore, 
\begin{equation} \label{EQ:d and N'}
n>1\ \text{\it square-free}\ \implies d(n)=-n'.
\end{equation}

In 2010 Lava \cite[p. 129]{BL} conjectured that Giuga numbers are the solutions of the differential equation $n'=n+1$. Grau and Oller-Marc\'en \cite{GO} proved in 2011 that Giuga numbers are the solutions of the differential equation $n'=an+1$, with $a\in\N$.

The following result shows that if $k$ and $n$ are Giuga numbers or primes, then the product $kn$ cannot be a Giuga number, and that the product of two primary pseudoperfect numbers cannot be another one. (In contrast, the product of a primary pseudoperfect number and a prime can be either a primary pseudoperfect number, e.g., $6 \cdot 7 = 42$, or a Giuga number, e.g., $6 \cdot 5 = 30$, or neither, e.g., $6 \cdot 11 = 66$---compare Theorem~\ref{PROP:pppGiuga}.)

\begin{theorem} \label{THM:never}
The product of two integers each of which is either a Giuga number or a prime is never a Giuga number, and the product of two primary pseudoperfect numbers is never a primary pseudoperfect number.
\end{theorem}

\begin{proof}
We show more generally that, {\it if $M>1$ and $n>1$ are coprime integers satisfying $d(M)\equiv \epsilon \pmod{M}$ and $d(n)\equiv \epsilon \pmod{n}$, where $\epsilon=\pm1$, then} $d(Mn)\not\equiv \epsilon \pmod{Mn}$. Indeed, Theorem \ref{thm: LeibnitzRule} part (ii) gives
\[
d(Mn)=Md(n)+nd(M)\equiv \epsilon (M+n)\pmod{Mn}
\]
and it follows that the congruence $d(Mn)\equiv \epsilon \pmod{Mn}$ holds only if $M =1$ or $n =1$, a~contradiction.
\end{proof}

\begin{proposition} \label{SubsetsOfPrimesEgyptian} Given a positive integer $n$, let $P$ be the set of its distinct prime divisors, and let $Q$ and $R$ be subsets of $P$ satisfying $Q \cup R = P$ and $Q \cap R = \emptyset$. Suppose that $(n,d_Q)$ and $(n,d_R)$ satisfy the congruences
\[
\sum_{ p \in Q} \frac{1}{p}+\frac{d_Q}{n} \equiv 1\equiv  \sum_{ p \in R} \frac{1}{p}+\frac{d_R}{n}  \pmod{1}.
\]
Then $d_Q$ and $d_R$ are related by $d_Q+d_R=d$, where $(n,d)$ is a solution to congruence \eqref{EQ:coolNums}. 
\end{proposition}

\begin{proof}
We have
$$d=d_Q+d_R \equiv -\sum_{ p \in Q} \frac{n}{p}- \sum_{ p \in R} \frac{n}{p}\equiv - \sum_{ p \in P} \frac{n}{p}\equiv - \sum_{ p \mid n} \frac{n}{p} \pmod{n}$$
and the result follows.
\end{proof}

An interesting variation on the Egyptian fraction equation (\ref{EQ:coolNums}) is obtained by replacing the integers in the definition with polynomials having integer coefficients. Let $n(x)=p_1(x) p_2(x) \cdots p_m(x)\in{\Z}[x]$, with $p_i(x) \in {\Z}[x]$ primitive and irreducible in ${\Q}[x]$ for each $i$. From now on, we will assume that polynomials denoted by $p(x)$ are prime in this sense. We seek $d(x) \in \Z[x]$ such that
\begin{align} \label{EQ:coolPolynomials}
\sum_{p(x) \mid n(x)} \frac{1}{p(x)}+\frac{d(x)}{n(x)} \equiv 1 \pmod{1}.
\end{align}
As before, solutions are given by
\[
d(x) \equiv -\sum_{p(x) \mid n(x)} \frac{n(x)}{p(x)} \pmod{n(x)}.
\]

\begin{example}
Take $n(x)=p_1(x) p_2(x)p_3(x)$, where the polynomials $p_1(x)=x, p_2(x)=-2x+1$ and $p_3(x)=-2x-1$ are prime. Then
\[
\frac{1}{x}+\frac{1}{-2x+1}+\frac{1}{-2x-1}+\frac{d(x)}{x(-2x+1)(-2x-1)}=
\frac{-1+d(x)}{x(-2x+1)(-2x-1)}.
\]
Consequently, $d(x) \equiv 1 \pmod{n(x)}$ is a solution to \eqref{EQ:coolPolynomials}. Thus, taking $x=p$ for some prime $p \in \Z$, if $-2p+1$ and $-2p-1$ are also prime, then $n(p)$ satisfies an equation akin to that of a primary pseudoperfect number, although the primes may be negative. For instance, we may take $p=19$, $-2p+1=-37$ and $-2p-1=-39$ to conclude that the number $27417=19\times-37 \times - 39$ is almost primary pseudoperfect:
\[
\frac{1}{19}+\frac{1}{-37}+\frac{1}{-39}+\frac{1}{27417}=0.
\]
\end{example}

To prove the square-freeness of $m,\frac{m+2}{2},2m+1,$ and $2m+3$, Moser \cite{Moser} showed that if $(m,n)$ is a solution of the Erd\H{o}s--Moser equation, then $(m,1)$, $(m+2,2)$, $(2m+1,2)$ and $(2m+3,4)$ are  solutions $(n,d)$ to the congruence~\eqref{EQ:coolNums}. We now aim to find an additional solution of the form $(n,d)=(m-1,x)$.

We employ the {\em Carlitz-von Staudt Theorem} \cite[Theorem~4]{Carlitz}, as corrected by Moree \cite[Theorem~3]{MoreeTopHat}.

\begin{theorem}[Carlitz-von Staudt]\label{thm: CarlitzVS} Let $n$ and $m$ be positive integers. Then
\[
S_n(m) \equiv \begin{cases} 
                                          \displaystyle-\!\!\!\sum_{\substack{p \mid m+1,\\ p-1 \mid n}} \dfrac{m+1}{p} \pmod{m+1}\quad\text{ if $n$ is even},\\
                                           0\pmod{m(m+1)/2} \qquad\quad\quad \text{if $n$ is odd}.
                                         \end{cases}
\]
\end{theorem}

\begin{proof}[Proof of the first case.]
When $n$ is even, apply Corollary \ref{PowerSumporder} to each factor $p^{v_p(m+1)}$ of $m+1$ and use the Chinese Remainder Theorem.
\end{proof}

\begin{theorem} \label{thm: Rabbit}
Let $(m,n)$ be a nontrivial solution to the Erd\H{o}s--Moser equation.\\
{\rm(i).} Let $$X:=\sum_{\substack{p \mid m-1,\\p-1 \nmid n}} \frac{m-1}{p}.$$
The pair $(n,d)=(m-1,2^n-1-X)$ satisfies congruence \eqref{EQ:coolNums}.\\
{\rm(ii).} If $p\mid m-1$, then $n=p-1+k\cdot \ord_p(2)$ for some $k\geq 0$.\\
{\rm(iii).} Given $p\mid m-1$, if $p^e \mid m-1$ with $e \geq 1$, then $p^{e-1} \mid 2^n-1$.\\
{\rm(iv).} Given $p\mid m-1$, if $p-1 \mid n$ and $p^{e} \mid 2^n-1$ with $e \geq 1$, then $p^{e+1} \mid m-1 $; in particular, $p^2 \mid m-1$.
\end{theorem}

\begin{proof}
(i). Rearranging the Erd\H{o}s--Moser equation, we have
\[
S_n(m-2)=(m+1)^n-m^n-(m-1)^n  \equiv 2^n-1 \pmod{m-1}.
\]
As in the proof of Theorem \ref{cor:EMCongruences}, the hypothesis implies $n$ is even. Hence, by the Carlitz-von Staudt Theorem,
\[
-\sum_{\substack{\ell \mid m-1,\\\ell-1 \mid n}} \frac{m-1}{\ell} \equiv 2^n-1 \pmod{m-1},
\]
where $\ell$ denotes a prime. By Proposition \ref{SubsetsOfPrimesEgyptian}, this proves (i).\\
(ii). If $p\mid m-1$, but $p-1\nmid n$, then reducing both sides modulo $p$ yields $2^n \equiv 1 \pmod{p},$
so that $n$ is a multiple of $\ord_p(2)$. Recall that $\ord_p(2) \mid p-1$. It follows that if $p \mid m-1$, then $n$ is a multiple of $\ord_p(2)$.

We now show that $n\ge p-1$. We refer to \cite[Lemma 6]{MoreeEtAl}, a result of Moser, which states that $3n \geq 2m$. This implies that $n \geq p-1$ and proves (ii).\\
(iii). By Proposition \ref{coolEquationEquivalentConditions},$$p^e \mid m-1 \implies p^{e-1} \mid 2^n-1-X.$$Since $X \equiv 0 \pmod{p^{e-1}}$, result (iii) follows.\\
(iv). Finally, assume that $p-1 \mid n$. We proceed by induction on $e \geq 1$. For the base case $e=1$, since $p-1 \mid n$ and $p \mid m-1$, we have $2^n-1-X \equiv 0 \pmod{p}$. By Proposition \ref{coolEquationEquivalentConditions}, the base case follows. Now assume (iv) for $e\geq1$. Then since $m-1 \equiv 0 \pmod{p^{e}}$ and $p-1 \mid n$, we get $2^n-1-X \equiv 0 \pmod{p^e}$. By Proposition~\ref{coolEquationEquivalentConditions}, the induction is complete.
\end{proof}

\begin{corollary}
If $(m,n)$ is a solution of the Erd\H{o}s--Moser equation with $m \equiv 1 \pmod{3}$, then in fact $m \equiv 1 \pmod{3^7}$.
\end{corollary}

\begin{proof}
It is known \cite{MoreeOnMoser} that $n$ is divisible by $2^8\cdot 3^5$. Therefore $\phi(3^6) \mid n$, and it follows that $2^n-1 \equiv 0 \pmod{3^6}$. Now Theorem \ref{thm: Rabbit} part (iv) implies $3^7 \mid m-1$.
\end{proof}

%%%
\section{Bernoulli numbers}\label{Bernoulli numbers}

In this section, we apply some of the results of previous sections to study the {\it Bernoulli numbers} $B_0,B_1,B_2,B_3,B_4,\dotsc=1, -1/2,1/6,0,-1/30,\dotsc$.

\begin{corollary} \label{COR:PascalBernoulli}
For $n \geq 1$ and every positive integer $m\le n$, we have the relation
\[
 \sum_{k=m-1}^{n-1} (-1)^k \binom{n}{k}  \binom{k+1}{m} \frac{B_{k+1-m}}{k+1}=(-1)^{m+1}\binom{n}{m}.
\]
\end{corollary}

\begin{proof}
By {\it Bernoulli's formula} (see, e.g., Conway and Guy \cite[pp. 106--109]{ConwayGuy}), the polynomial
\begin{equation} \label{EQ:Faulhaber}
P_n(x):=\frac{1}{n+1} \sum_{j=0}^n (-1)^j \binom{n+1}{j} B_j x^{n+1-j}
\end{equation}
satisfies
\begin{equation} \label{EQ:Bernoulli}
S_n(a)=P_n(a)
\end{equation}
for any positive integers~$n$ and~$a$. Substituting this into Pascal's identity \eqref{EQ:Pascal}, we expand the right-hand side and get
\[
\sum_{k=0}^{n-1} \binom{n}{k} \frac{1}{k+1} \sum_{j=0}^k (-1)^j \binom{k+1}{j} B_j a^{k+1-j}=\sum_{m=1}^{n} \binom{n}{m}a^m.
\]
Setting $n=k+1-j$, we can write this as
\[
\sum_{k=0}^{n-1} \sum_{n=1}^{k+1} (-1)^{k+1-n} \binom{n}{k} \binom{k+1}{n} \frac{B_{k+1-n}}{k+1} a^{n}=\sum_{m=1}^{n} \binom{n}{m}a^m.
\]
Since this holds for all $a>0$, we may equate coefficients when $n=m$, and the desired formula follows.
\end{proof}

In particular, the case $m=1$ is
\[
 \sum_{k=0}^{n-1} (-1)^{k} \binom{n}{k} B_{k}=n.
\]
Since $B_1=-1/2$ and $B_{2n+1}=0$ for $n>0$, this case is equivalent to
\begin{equation}
 \sum_{k=0}^{n-1} \binom{n}{k} B_{k}=0, \label{EQ:recursion}
\end{equation}
which is the standard recursion for the Bernoulli numbers.
Thus, Corollary \ref{COR:PascalBernoulli} is a generalization of this recursion.
%(The reader may verify that one obtains essentially the same generalization if instead of Pascal's identity one uses its analog for even exponents, Theorem~\ref{THM:evenPascal}.)
 
As a numerical example, take $n=8$ and $m=3$:
\begin{align*}
  \sum_{k=2}^{6}  (-1)^k  \binom{8}{k} \binom{k+1}{m} \frac{B_{k-2}}{k+1}&= \frac{28}{3}B_{0}-56 B_{1}+140B_{2}-\frac{560}{3}B_3+ 140B_{4}\\
&=  \frac{28}{3} + 28 + \frac{70}{3}-0 - \frac{14}{3}=56=\binom{8}{3},
\end{align*}
as predicted.

%Replace $S_{2k}(a)$ in Theorem \ref{THM:evenPascal} with its interpolating polynomial
%$$S_{2k}(a)=\frac{1}{2k+1} \sum_{j=0}^{2k} (-1)^j \binom{2k+1}{j} B_j a^{2k+1-j}.$$
%Expanding the binomial in Theorem \ref{THM:evenPascal} then yields
%\[
%2 \!\!\sum_{k=0}^{(n-2)/2} \binom{n}{2k} \frac{1}{2k+1} \sum_{j=0}^{2k} (-1)^j \binom{2k+1}{j} B_j a^{2k+1-j}=\sum_{m=1}^{n-1}  \binom{n}{m} a^m,
%\]
%which we can write as
%\[
%2 \!\!\sum_{k=0}^{(n-2)/2} \binom{n}{2k} \frac{1}{2k+1} \sum_{n=1}^{2k+1} (-1)^{n+1} \binom{2k+1}{n} B_{2k+1-n} a^n=\sum_{m=1}^{n-1}  \binom{n}{m} a^m.
%\]
%Comparing coefficients when $n=m$, we obtain the following identity.

\smallskip

\begin{corollary} \label{COR:evenPascal}
Let $n \geq 2$ be even and let $m<n$ be a positive integer. Then
\[
 \sum_{k=\lceil (m-1)/2\rceil}^{(n-2)/2} \binom{n}{2k} \binom{2k+1}{m} \frac{B_{2k+1-m}}{2k+1}=(-1)^{m+1} \frac{1}{2}\binom{n}{m},
\]
where $\lceil .\rceil$ denotes the ceiling function.
\end{corollary}

\begin{proof}
We follow the steps in the previous proof, except that instead of Pascal's identity we use its analog for even exponents, Theorem~\ref{THM:evenPascal}. Details are omitted.
\end{proof}

For example, again take $n=8$ and $m=3$:
\begin{align*}
%2\sum_{k=1}^{3} \binom{8}{2k} \binom{2k+1}{3} \frac{B_{2k-2}}{2k+1}&=2\left( \frac{28}{3}B_{0}+140 B_{2}+140B_{4}\right)\\
%&= 2\left( \frac{28}{3} + \frac{70}{3} - \frac{14}{3}\right)=56=\binom{8}{3},
 \sum_{k=1}^{3} \binom{8}{2k} \binom{2k+1}{3} \frac{B_{2k-2}}{2k+1}&=\frac{28}{3}B_{0}+140 B_{2}+140B_{4}\\
&= \frac{28}{3} + \frac{70}{3} - \frac{14}{3}=28=\frac{1}{2}\binom{8}{3},
\end{align*}
also as predicted.
%If $m$ is even, then since $B_{2n+1}=0$ for $n>0$, the identity is simply
%$$ \frac{2}{m+1} \binom{n}{m} \binom{m+1}{m}  (-\frac{1}{2})=- \binom{n}{m}.$$

%\textbf{Edited - J
Comparing the numerical examples for Corollaries \ref{COR:PascalBernoulli} and \ref{COR:evenPascal}, one sees that Corollary~\ref{COR:evenPascal} follows from Corollary~\ref{COR:PascalBernoulli}, together with the standard recursion~\eqref{EQ:recursion} solved for $B_1$.

Let us now adopt Kellner's notation \cite{Kellner2} and write the Bernoulli numbers as
$$ B_k = \frac{n_k}{D_k}$$
in lowest terms with $D_k > 0$. Thus,
$$\frac{n_0}{D_0}=\frac{1}{1}, \frac{n_1}{D_1}=\frac{-1}{2},\frac{n_3}{D_3}=\frac{n_5}{D_5}=\frac{n_7}{D_7}=\frac{n_9}{D_9}=\dotsb=\frac{0}{1},$$
and
$$\frac{n_{2n}}{D_{2n}}=\frac{1}{6}, \frac{-1}{30}, \frac{1}{42}, \frac{-1}{30}, \frac{5}{66}, \frac{-691}{2730}, \frac{7}{6},\frac{-3617}{510}, \frac{43867}{798}, \frac{-174611}{330},\frac{854513}{138},\frac{-236364091}{2730}, \dotsc,$$
for $n=1,2,3,4,5,6,7,8,9,10,11,12,\dotsc,$ respectively.

Recall that the {\em von Staudt-Clausen Theorem} states that, for $n\ge1$,
\begin{equation}\label{EQ:Staudt-Clausen}
\sum_{p-1 \mid 2n} \frac{1}{p}+B_{2n} \equiv 1 \pmod{1}.
\end{equation}
As a consequence, the denominator of $B_{2n}$ is the square-free number $D_{2n} = \prod_{p-1 \mid 2n} p$. Then multiplying \eqref{EQ:Staudt-Clausen} by $D_{2n}$ gives
$$n_{2n} \equiv -\sum_{p\mid D_{2n}} \frac{D_{2n}}{p} \pmod{D_{2n}}.$$
It now follows from the definition of $d(n)$ in \eqref{EQ:d(n)} that {\em the numerator of $B_{2n}$ satisfies}
%Moreover, in the notation \eqref{EQ:d(n)} for $d(n)$, the numerator satisfies
%$$[\numer(B_{2n})]_{\denom(B_{2n})}=d(\denom(B_{2n}));$$
%equivalently, in terms of the arithmetic derivative,
$$n_{2n} \equiv d(D_{2n})\pmod{D_{2n}}.$$

\begin{theorem}
Let $n$ and $k$ be positive integers. For the difference $B_{2nk}-B_{2n}$,\\
{\rm(i).} the denominator equals
$$\denom(B_{2nk}-B_{2n}) = \frac{D_{2nk}}{D_{2n}}\in\N,$$
{\rm(ii).} and the numerator satisfies the congruence
%\[
%[\numer(B_{2nk}-B_{2n})]_{\denom(B_{2nk}-B_{2n})}=d(\denom(B_{2nk}-B_{2n})).
%\]
%or equivalently,
\[
\numer(B_{2nk}-B_{2n})\equiv d(\denom(B_{2nk}-B_{2n})) \pmod{\denom(B_{2nk}-B_{2n})}.
\]
%that is,
%\[
%\numer(B_{2nk}-B_{2n})\equiv -\sum_{p\mid\denom(B_{2nk}-B_{2n})} \frac{\denom(B_{2nk}-B_{2n})}{p} \pmod{\denom(B_{2nk}-B_{2n})}.
%\]
\end{theorem}

\begin{proof}
(i). For any $m\in\N$, the von Staudt-Clausen Theorem gives $B_{2m}=A_{m}-\sum_{p-1 \mid 2m} \frac{1}{p}$, where $A_{m} \in \Z$. Hence
\begin{align} \label{BernoulliDiff}
B_{2nk}-B_{2n}=A_{nk}-A_{n}-\left(\sum_{p-1 \mid 2nk}\frac{1}{p}-\sum_{p-1 \mid 2n}\frac{1}{p}\right)=A_{nk}-A_{n}-\sum_{\substack{p-1 \mid 2nk,\\p-1 \nmid 2n}}\frac{1}{p}.
\end{align}
Therefore,
$$\denom(B_{2nk}-B_{2n}) =\prod_{\substack{p-1 \mid 2nk,\\p-1 \nmid 2n}}\! p =\frac{\prod_{p-1 \mid 2nk} p}{\prod_{p-1 \mid 2n} p}= \frac{D_{2nk}}{D_{2n}}\in\N.$$
(ii). Writing $\frac{P}{Q}:=B_{2nk}-B_{2n}$, we have, by part (i) and equation (\ref{BernoulliDiff}), 
\[
 \sum_{p\mid Q}\frac{1}{p}+\frac{P}{Q}=\sum_{\substack{p-1 \mid 2nk,\\p-1 \nmid 2n}}\frac{1}{p}+\frac{P}{Q}  \equiv 1 \pmod{1}.
\]
Since $d(Q)=-\sum_{p\mid Q}\frac{Q}{p}$, we obtain $P\equiv d(Q)\pmod{Q}$, proving (ii).
\end{proof}

For example, taking $n=1$ and $k=12$, we have
$$B_{24}-B_{2}=\frac{-236364091}{2730} - \frac{1}{6}=\frac{-39394091}{455}.$$
From Theorem~\ref{thm: LeibnitzRule} part (ii) and equation \eqref{EQ:d(p)}, we compute that $d$ of the denominator equals
\[
d(455)= d(5\cdot7\cdot13)=-5\cdot7-5\cdot13-7\cdot13 = -191.
\]
These calculations agree with (i) and (ii), which in this example state that
$$ \denom(B_{24}-B_{2})= \frac{D_{24}}{D_{2}}=\frac{2730}{6}=455$$
and that $- 39394091 \equiv d(455)\pmod{455}$.\\

%{\bf Edited - J. What is the relation between Theorems~\ref{thm: agoh} and \ref{thm: pseudo}?}

Here is a result due to Agoh \cite{Agoh} (see also \cite[pp. 41, 49]{BBBG} and \cite{Kellner}).

\begin{theorem}[Agoh] \label{thm: agoh} The following statements about a positive integer $n$ are equivalent: \\
{\rm (i).} $p \mid (\frac{n}{p}-1)$, for each prime factor $p$ of $n$. \\
{\rm (ii).} $S_{n-1}(n-1) \equiv -1 \pmod{n}$. \\
{\rm (iii).} $n B_{n-1} \equiv -1 \pmod{n}$.
\end{theorem} 

We prove a related result, using a theorem of Kellner.

\begin{theorem} \label{thm: pseudo} 
{\rm (i).} Let $n$ and $d$ be positive integers, with $n$ square-free. Then $p \mid (\frac{n}{p}+d)$, for each prime factor $p$ of $n$, if and only if $ S_{\phi(n)}(n) \equiv d \pmod{n}$. \\
{\rm (ii).} For {\em any} positive integer $n$, we have the congruence
$$S_{\phi(n)}(n) \equiv n B_{\phi(n)} \pmod{n}.$$
\end{theorem} 

\begin{proof}
(i). The statement holds for $n=1$. Now take $n>2$, let $p$ be a prime factor of~$n$, and set $n=pq$. Then using
Lemma \ref{LEM:linear} we have
\[
\sum_{j=1}^{n}j^{\phi(n)} \equiv  q \sum_{j=1}^{p}j^{\phi(n)} \equiv q \sum_{j=1}^{p-1}j^{\phi(n)}  \pmod{p}.
\]
Since $n$ is square-free, $\gcd(p,q)=1$ and so $\phi(n)=\phi(p)\phi(q)$. Thus  $\phi(n)$ is divisible by  $\phi(p)=p-1$, and hence by Fermat's little theorem,
\[
q \sum_{j=1}^{p-1}j^{\phi(n)} \equiv q(p-1)\equiv -q \pmod{p}. 
\]
As $q=n/p$, we get
\begin{equation}\label{eq:agoh}
\text{prime } p\mid n \implies \sum_{j=1}^{n}j^{\phi(n)} \equiv - \frac{n}{p} \pmod{p}.
\end{equation}

To prove (i), assume first that $p \mid (\frac{n}{p}+d)$ for all primes $p\mid n$, so that $- \frac{n}{p} \equiv d \pmod{p}$. Together with \eqref{eq:agoh} and the square-freeness of $n$, this implies that $\sum_{j=1}^{n}j^{\phi(n)} \equiv d \pmod{n}$. Conversely, if the latter holds, then \eqref{eq:agoh} yields $-\frac{n}{p} \equiv d \pmod{p}$. This proves (i).

(ii). It is easy to see that (ii) holds if $n=1$ or $2$. Now take $n \geq 3$ and recall that then $\phi(n)$ is even.
For any $n,m \in \N$ with $n$ even, Kellner \cite[Theorem 1.2]{Kellner} proved that
\[
S_n(m) \equiv (m+1) B_{n} \pmod{m+1}.
\]
Setting $n= \phi(n)$ and $m=n-1$, part (ii) follows.
\end{proof}

When $n>3$ is prime, we can improve part (ii) to a supercongruence.

\begin{theorem} \label{thm: prime super} 
If $p>3$ is prime, then
$$S_{p-1}(p) \equiv p B_{p-1} \pmod{p^3}.$$
\end{theorem}
\begin{proof}
Bernoulli's formula \eqref{EQ:Bernoulli} gives $S_{p-1}(p-1) = P_{p-1}(p-1)$. For prime $p>3$, the von Staudt-Clausen Theorem \eqref{EQ:Staudt-Clausen} implies that $P_{p-1}(p-1)  \equiv p B_{p-1} \pmod{p^3}$ (for details, see the proof of \cite[Theorem~1]{Sondow}, where $P_{p-1}(p-1)$ is written symbolically as $(B+p)^p/p$). As $S_{p-1}(p) \equiv S_{p-1}(p-1) \pmod{p^3}$, this proves the theorem.
\end{proof}

\section{Moser's Mathemagical Rabbits}\label{Moser's Magical Rabbits}
In this section, we reveal some of the magic behind Moser's ``mathemagical rabbits'' \cite{MoreeTopHat}. In particular, we give a hint as to why one could expect $m,\frac{m+2}{2},2m+1$, and $2m+3$ to be square-free. 
Consider the generalized Erd\H{o}s--Moser equation:
\[
S_n(m)=a(m+1)^n \iff (a+1) S_n(m) = a S_n(m+1).
\]
Let $P_n(x) \in \Q[x]$ denote the polynomial interpolating $S_n$ in \eqref{EQ:Faulhaber}. Then
\[
(a+1) P_n(m)=a P_n(m+1).
\]
Let $L_n \in \Q$ satisfy the conditions that
\[
L_n P_n(x) \in \Z[x]
\]
and that the greatest common divisor of the coefficients of $L_n P_n(x)$ is $1$.
Set $Q_n(x) := L_n P_n(x)$. Then
\[
(a+1) Q_n(m)=a Q_n(m+1).
\]
On the other hand, it is known that $P_n(x)$ is given by \eqref{EQ:Faulhaber}.
For $j=1,2,\ldots,n$, let 
\[
R_{j}=R_{j}(n):=\frac{D_{j}}{\gcd(D_{j},\binom{n+1}{j})} \in\N.
\]
Then
\[
L_n=(n+1) \lcm(R_1,R_2,\ldots,R_n)
\]
and we obtain
\[
Q_n(x)=\lcm(R_1,R_2,\ldots,R_n) \sum_{j=0}^n (-1)^j \binom{n+1}{j} B_j x^{n+1-j}.
\]
We now focus on the Erd\H{o}s--Moser equation, when $a=1$ and $n$ is even, i.e., a counterexample to the Erd\H{o}s--Moser conjecture:
\[
2 Q_n(m)=Q_n(m+1).
\]
In this case, Corollary \ref{FactorsOfPowerSumPolynomial} implies $m(m+1)(2m+1)$ divides $Q_n(m)$, and $(m+1)(m+2)(2m+3)$ divides $Q_n(m+1)$. Note the appearance of the numbers $m,m+2,2m+1,2m+3$ as divisors---these are the same numbers that appear in Moser's trick.

Consider $Q_n(m+1)$ modulo $m$:
\[
0 \equiv Q_n(m+1)=\lcm(R_1,R_2,\ldots,R_n) \sum_{j=0}^n (-1)^j \binom{n+1}{j} B_j (m+1)^{n+1-j} \pmod{m}
\] 
\[
\equiv \lcm(R_1,R_2,\ldots,R_n) \sum_{j=0}^n (-1)^j \binom{n+1}{j} B_j =(n+1) \lcm(R_1,R_2,\ldots,R_n) = L.
\]
Therefore $m$ divides $L$. The denominators of Bernoulli numbers are square-free, so we almost obtain another proof of the square-freeness of $m$.

\bigskip
\noindent{\bf Acknowledgments}. The authors are very grateful to Wadim Zudilin for many helpful suggestions on the terminology and exposition of the first half of the paper.

The second author was supported by the National Science Foundation Graduate Research Fellowship under Grant No. DGE 1106400. Any opinion, findings, and conclusions or recommendations expressed in this material are those of the authors and do not necessarily reflect the views of the National Science Foundation.

\end{document}